\documentclass{amsart}
\usepackage[numbers,sort]{natbib}
\usepackage{color,mathtools, verbatim, esint, amssymb}
\usepackage{todonotes}
\mathtoolsset{showonlyrefs}

\newcommand\blfootnote[1]{%
  \begingroup
  \renewcommand\thefootnote{}\footnote{#1}%
  \addtocounter{footnote}{-1}%
  \endgroup
}

\newcommand{\ddb}{\sqrt{-1}\partial\overline{\partial}}

\renewcommand{\[}{\begin{equation} \begin{aligned} }
      \renewcommand{\]}{\end{aligned} \end{equation}}
\renewcommand{\phi}{\varphi}

\newtheorem{thm}{Theorem}
\newtheorem{prop}[thm]{Proposition}
\newtheorem{lem}[thm]{Lemma}
\newtheorem{cor}[thm]{Corollary}

\theoremstyle{definition}
\newtheorem{defn}[thm]{Definition}

\author{Julius Ross, G\'abor Sz\'ekelyhidi}
\address{Mathematics, Statistics and Computer Science, University of
  Illinois at Chicago, Chicago IL, USA}
\email{julius@math.uic.edu}
\address{Department of Mathematics, University of Notre Dame, Notre
  Dame IN, USA}
\email{gszekely@nd.edu}
\title{Twisted K\"ahler-Einstein metrics}
\dedicatory{Dedicated to D. H. Phong on the occasion of his 65$^{th}$ birthday.} 
\date{}

\begin{document}
\begin{abstract}
We prove an existence result for twisted K\"ahler-Einstein
metrics, assuming an appropriate twisted K-stability condition. An
improvement over earlier results is that certain non-negative twisting
forms are allowed.  
\end{abstract}
\maketitle

\blfootnote{The first author is supported in part by NSF grants DMS-1707661 and DMS-1749447.  The second author is supported in part by NSF grant DMS-1350696}

\section{Introduction}
Let $M$ be a Fano manifold, together with a line bundle $T\to M$. 
Let $\beta\in c_1(T)$ be a smooth non-negative form that can be expressed
as an average
\[ \label{eq:bdecomp}\beta = \int_{|T|} [D]\, d\mu(D), \]
where $d\mu$ is a volume form on the linear system $|T|$. A typical
example is obtained if $|T|$ is basepoint free, and $\beta$ is the
pullback of the Fubini-Study metric under the corresponding map
$M\to\mathbf{P}^N$ (see \cite[Theorem 19]{LSz15}). More generally we could allow the divisors $D$ to
be in the linear system $|kT|$ for some $k > 1$, but for simplicity of
notation we will only consider the case $k=1$. 

Our goal is to study the existence of solutions to the equation
\[ \mathrm{Ric}(\omega) = \omega + \beta \]
on $M$. We necessarily have $\omega\in c_1(L)$, where $L=K^{-1}\otimes
T^{-1}$ in terms of the canonical bundle $K$ of $M$. We call a
solution $\omega$ of this equation a twisted 
K\"ahler-Einstein metric on $(M,\beta)$.  The main result is the following.
\begin{thm}\label{thm:main}
  There exists a twisted K\"ahler-Einstein metric on $(M,\beta)$ if
  $(M,\beta)$ is K-stable. 
\end{thm}
We will define K-stability of the pair $(M,\beta)$ in
Section~\ref{sec:Kstability} below. Note that if $T$ is trivial, so
that $\beta=0$, then $L=K^{-1}$, and we are seeking a
K\"ahler-Einstein metric on $M$. In this case Theorem~\ref{thm:main} 
was proven by Chen-Donaldson-Sun~\cite{CDS} in solving the Yau-Tian-Donaldson
conjecture~\cite{Yau93, Tian97, D-Toric}. When $\beta\in c_1(M)$ is strictly positive,
Datar and the second author~\cite{DSz} showed a slightly weaker statement,
namely that if $(M,\beta)$ is K-stable, then for any $\epsilon > 0$
there is a solution of the equation $\mathrm{Ric}(\omega) = \omega +
(1+\epsilon)\beta$. This is more or less equivalent to replacing
``K-stable'' by ``uniformly K-stable'' in the statement of
Theorem~\ref{thm:main}. In much more generality, allowing positive
currents $\beta$, the result assuming uniform K-stability was also shown by
Berman-Boucksom-Jonsson~\cite{BBJ}, using very different
techniques. In the setting when $\beta\in c_1(M)$ is the current of
integration along a smooth divisor, the statement of
Theorem~\ref{thm:main} was also shown by
Chen-Donaldson-Sun~\cite{CDS}, where instead of twisted
K\"ahler-Einstein metrics, one considers K\"ahler-Einstein metrics
with cone singularities along the divisor.  Let us also remark
that it would be natural to extend Theorem~\ref{thm:main} to pairs
$(M,\beta)$ that admit automorphisms, using a suitable 
notion of K-polystability rather than K-stability. This would not
introduce substantial new difficulties, however in this paper we focus on
the case of no automorphisms to simplify the discussion.

In Section~\ref{sec:Kstability} we will give the definition of
K-stability of a pair $(M,\beta)$, which is similar to
log-K-stability~\cite{Li15} and twisted K-stability~\cite{Der14}. In 
the case when $\beta$ is the pullback of a positive form by a
map, stability of the pair is related to the stability of the map in
the sense of \cite{DR17}. We
then prove Theorem~\ref{thm:main} in Section~\ref{sec:mainproof} along
the lines of the argument in \cite{DSz}. An important simplification
of the prior arguments in Chen-Donaldson-Sun~\cite{CDS} as well as
\cite{Sz13_1, DSz} is provided by the work of the second author and
Liu~\cite{LiuSz} on Gromov-Hausdorff limits of K\"ahler manifolds with
only lower bounds on the Ricci curvature, rather than a two-sided
bound as in Donaldson-Sun~\cite{DS1}. An additional observation, given
in Corollary~\ref{cor:notKstable} below, allows us to obtain the
existence of a twisted K\"ahler-Einstein metric under the assumption
of K-stability, rather than the stronger uniform K-stability which
would follow more directly from the methods of \cite{DSz}. \\

\section{K-stability}\label{sec:Kstability}
Let $M, T, \beta$ be as in the introduction, and $L=K^{-1}\otimes
T^{-1}$. Note that since $M$ is Fano, the line bundles $T, L$ are
uniquely determined by $\beta$, given that $\beta\in c_1(T)$. 
In this section we discuss K-stability of $(M,\beta)$, and
prove some basic properties.
First we have the following definition, which agrees with that in
Tian~\cite{Tian97} when $T$ is the trivial bundle so that $\beta=0$. 
\begin{defn}
  A special degeneration for $(M,L)$ of exponent $r > 0$ consists of
  an embedding $M\subset\mathbf{P}^N$ using a basis of sections of
  $L^r$, together with a $\mathbf{C}^*$-action $\lambda$ on
  $\mathbf{P}^N$, such that the limit $\lim_{t\to 0} \lambda(t)\cdot
  M$ is a normal variety. 
\end{defn}

We will refer to a special degeneration by the $\mathbf{C}^*$-action
$\lambda$, leaving implicit the projective embedding of $M$ that is
also part of the data. 
Next, we define the Donaldson-Futaki invariant $DF(M,\lambda)$ in the
same way as in Donaldson~\cite{D-Toric}, in terms of the weights of
the action on the spaces of sections $H^0(M, L^{kr})$ as
$k\to\infty$. In addition we will need a differential geometric formula for the
Donaldson-Futaki invariant. For this let $Z=\lim_{t\to 0}
\lambda(t)\cdot M$. We can assume that the $S^1$-subgroup of $\lambda$
acts through $SU(N+1)$, and so we have a Hamiltonian function $\theta$
on $\mathbf{P}^N$ generating $\lambda$. 
\begin{prop}\label{prop:DFformula}
  Let $\omega$ denote the restriction of the Fubini-Study metric to
  $Z$. We then have
  \[ DF(M, \lambda) = -V^{-1}\int_Z \theta\,(n\mathrm{Ric}(\omega|_Z) -
    \hat{R}\omega)\wedge \omega^{n-1}, \]
  where $V$ is the volume of $Z$, and $\hat{R}$ is the average scalar
  curvature, so that the integral above is unchanged by adding a
  constant to $\theta$. 
\end{prop}
\begin{proof}
  Let us denote by $\omega_s$ the restriction of the Fubini-Study
  metric on $\lambda(e^{-s})\cdot M$. We thus have a family of metrics
  $\omega_s = \omega_0 + \ddb\phi_s$ 
  on $M$ in a fixed K\"ahler class. Since the central fiber $Z$ of our
  degeneration is normal, 
  the Donaldson-Futaki invariant $DF(M,\lambda)$ is given by the
  asymptotic derivative of the Mabuchi functional~\cite{Mab86} along
  this family $\omega_s$ (see Paul-Tian~\cite[Corollary 1.3]{PT}). I.e. we have
  \[ DF(M, \lambda) = \lim_{s\to\infty} -V^{-1} \int_M \dot\phi_s
    (n\mathrm{Ric}(\omega_s) - \hat{R}\omega_s)\wedge\omega_s^{n-1}.\]
  In addition we have $\dot\phi_s = \theta$ under identifying $M$ with
  $\lambda(e^{-s})\cdot M$. It therefore remains to show that these
  integrals on $M$ converge to the corresponding integral on $Z$.

  If $Z$ were smooth, then this convergence would be immediate. It is
  thus enough to show that the singularities of $Z$ do not contribute
  to the limit. For this, note first that we have a uniform upper
  bound $\mathrm{Ric}(\omega_s) < C\omega_s$ for the Ricci
  curvatures,  where $C$ depends on the curvature of the Fubini-Study
  metric, since curvature decreases in holomorphic
  subbundles. We can view $C\omega_s -
  \mathrm{Ric}(\omega_s)$ as a 
  positive current of dimension $(n-1,n-1)$, supported on
  $\lambda(e^{-s})\cdot M$. As $s\to\infty$, these currents converge
  (along a subsequence if necessary)  weakly to a limit current $T$,
  supported on $Z$. On the regular part of $Z$, this limit current is
  necessarily given by $C\omega - \mathrm{Ric}(\omega)$ in terms of
  the Fubini-Study metric $\omega$, and since the codimension of the
  singular set is at least 2, this determines $T$. 
\end{proof}

We are now ready to define the twisted Futaki invariant of the special
degeneration.
\begin{defn}
  Suppose that we have a special degeneration $\lambda$ for $M$ with
  Hamiltonian $\theta$ as
  above, and $Z=\lim_{t\to 0}\lambda(t)\cdot M$. Under the assumption
  \eqref{eq:bdecomp} we
  have an induced current $\gamma =
\lim_{t\to 0} \lambda(t)_*\beta$ on $Z$. The twisted Futaki invariant 
of this special degeneration is then defined to be
\[ \mathrm{Fut}_{\beta}(M, \lambda) = DF(M, \lambda) +  nV^{-1} \int_Z
  \theta\,(\gamma - c\omega_{FS})\wedge \omega_{FS}^{n-1}, \]
where $c$ is a constant so that
the expression is invariant under adding a constant to $\theta$. 
\end{defn}

Given this, we define K-stability of $(M,\beta)$ as follows. 
\begin{defn}
  The pair $(M,  \beta)$ is  K-stable, if
  $\mathrm{Fut}_\beta(M,\lambda) \geq 0$ for all special
  degenerations for $(M,L)$, with equality only if $\lambda$ is trivial. 
\end{defn}

It will be important for us to replace the smooth form $\beta$ with
currents of integration along divisors. The definition of the twisted
Futaki invariant above applies in this case too, leading to log-K-stability
(see Donaldson~\cite{DonaldsonCone}, Li~\cite{Li15}), and we will need
to compare these two notions. As in \cite{DSz}, the twisted Futaki
invariant with a smooth form $\beta$ is the same as the twisted Futaki
invariant using a generic divisor in the same class. This follows
from the decomposition \eqref{eq:bdecomp}, together with the following
result from Wang~\cite[Theorem 26]{WangMoment}.

\begin{prop}\label{prop:Wang}
  Let $D\subset \mathbf{P}^N$ have dimension $n-1$,
  and $\lambda$ a $\mathbf{C}^*$-action
  with Hamiltonian $\theta$ as above. Suppose that $\theta$ is
  normalized to have zero average on $\mathbf{P}^N$. Let $D_0 =
  \lim_{t\to 0} \lambda(t)\cdot D$, and denote by $w(D_0, \lambda)$
  the weight of the induced action on the Chow line over $D_0$. Then
  (up to a multiplicative normalization constant)
  \[ w(D_0, \lambda) = -\int_{D_0} \theta\,\omega^{n-1}. \]
\end{prop}

Under a projective embedding of the Chow variety, we can view each $D$
as in this proposition as a line in a vector space $V$ spanned by a
vector $v_D$. The weight $w(D_0, \lambda)$ is determined by the lowest
weight in the weight decomposition of $v_D$ under the
$\mathbf{C}^*$-action $\lambda$. It follows that as $D$ varies in a
linear system as in \eqref{eq:bdecomp}, there will be a hyperplane
section $H\subset |T|$ such that the corresponding weights will all be equal for
$D\not\in H$.  More precisely we have the following. 
\begin{prop}\label{prop:gw}
  Given any $\mathbf{C}^*$-action $\lambda$ with Hamiltonian $\theta$
  on $\mathbf{P}^N$, there
  is a hyperplane $H\subset |T|$ such that for all $D\in  |T|$
  we have
  \[ \label{eq:a2} \lim_{t\to 0} \int_{\lambda(t)\cdot D} \theta \omega^{n-1} \leq
    \lim_{t\to 0} \int_{\lambda(t)\cdot M} \theta\,(\lambda(t))_*\beta\wedge
    \omega^{n-1}, \]
  with equality for $D\in |T| \setminus H$. In addition, given an
  action of a torus $\mathbf{T}$, we can choose a $D\in |T|$ such that
  equality holds above for all $\lambda\subset \mathbf{T}$. 
\end{prop}
\begin{proof}
(Compare \cite[Lemma 9]{LSz15}.)  Using \eqref{eq:bdecomp} the equation \eqref{eq:a2} is true when averaged over
  $|T|$, i.e.  we have
  \[ \int_{|T|} \lim_{t\to 0} \int_{\lambda(t)\cdot D} \theta
    \omega^{n-1}\,d\mu(D) = 
    \lim_{t\to 0} \int_{\lambda(t)\cdot M} (\lambda(t))_*\beta\wedge
    \omega^{n-1}. \]
  At the same time by Proposition~\ref{prop:Wang},
  up to a normalizing constant, the limit on the left
  hand side of \eqref{eq:a2} is a Chow weight in geometric invariant
  theory. In particular it is given by the minimal weight under
  the weight decomposition of  the vector corresponding to $D$ in the
  Chow variety,   under the $\mathbf{C}^*$-action $\lambda$.
 Generically,
  i.e. on the complement of a hyperplane (corresponding to the
  vanishing of the lowest weight component), this weight will achieve
  its minimum and is   independent of $D$.

  For the second statement in the Proposition, we can take a generic
  $D$ that has a non-zero component in all the weight spaces which
  appear under the action of $\mathbf{T}$ on elements in $|T|$. 
\end{proof}

This result leads to an important finiteness property of special
degenerations inside a fixed projective space. We first have the
following (that is essentially a standard piece of Geometric Invariant Theory).

\begin{lem}\label{lem:finite}
  Fix $r > 0$.   There is a finite set $\mathcal{F}\subset \mathbf{R}$
  with the following
  property. Suppose that we have a special degeneration $\lambda$ of exponent
  $r$ for $M$, and a divisor $D\in |T|$ on $M$ such that the limit
  $(M_0, D_0)$ of the pair $(M,D)$ under $\lambda$ is not fixed by 
  any $\mathbf{C}^*$ subgroup of $SL(N+1)$ commuting with $\lambda$,
  apart from $\lambda$ itself (i.e. the centralizer of $\lambda$ in
  the stabilizer group is just $\lambda$). Let $\theta$ be
  the Hamiltonian for $\lambda$ normalized to have zero average on
  $\mathbf{P}^N$, and let $\Vert \lambda\Vert$ denote the $L^2$-norm
  of $\theta$ on $\mathbf{P}^N$. Then the normalized twisted Futaki
  invariant $\Vert\lambda\Vert^{-1} \mathrm{Fut}_{D}(M,\lambda)$ lies in
  $\mathcal{F}$. 
\end{lem}
\begin{proof}
  Note first of all that since any $\mathbf{C}^*$-subgroup can be
  conjugated into a maximal 
  torus of $SL(N+1)$,  up to moving the pair $(M,D)$ in its orbit, we
  can assume that $\lambda$ is in a fixed maximal torus
  $\mathbf{T}$. Then if $(M_0, D_0)$ is as in the statement of the
  Lemma, the normalized twisted Futaki invariant is determined by the pair
  $(M_0, D_0)$, since the induced $\mathbf{C}^*$-action is uniquely
  determined up to scaling. 

  The pair
  $(M_0,D_0)$ is represented by a point in a product of Chow varieties,
  i.e. under a projective embedding by a line spanned by a vector $v$
  in a vector space $V$ admitting a $\mathbf{T}$-action. Under the
  decomposition of $V$ into weight spaces for the $\mathbf{T}$-action,
  the weights appearing in the decomposition of $v$ must lie in a
  codimension-one affine subspace of $\mathfrak{t}^*$ by the
  assumption that $(M_0,D_0)$ has a one dimensional stabilizer in
  $\mathbf{T}$. The normalized twisted Futaki invariant is determined
  by this affine subspace rather than the components of $v$ in each
  corresponding weight space. Since there are only a finite number of
  possible such affine subspaces, we can have only finitely many
  different normalized twisted Futaki invariants. 
\end{proof}

\begin{cor}\label{cor:notKstable}
  Fix $r > 0$.  Suppose that for any $\epsilon > 0$ we have a special
  degeneration $\lambda$ of exponent $r$ for $(M,L)$ such that
  $\Vert\lambda\Vert^{-1} \mathrm{Fut}_\beta(M,\beta) <
  \epsilon$. Then $(M,\beta)$ is not K-stable. 
\end{cor}
\begin{proof}
  Given a special degeneration $\lambda$, we will show that we can
  either find another special degeneration with non-positive twisted
  Futaki invariant, or we can find a special degeneration $\lambda'$
  to which Lemma~\ref{lem:finite} applies, and which has smaller
  normalized twisted Futaki invariant than $\lambda$. If $\epsilon$ is
  sufficiently small, this will necessarily be non-positive. 

  By conjugating, we can assume that $\lambda$ is in a fixed maximal
  torus $\mathbf{T}$. By Proposition~\ref{prop:gw}, we can choose a
  $D\in |T|$, such that the twisted Futaki invariant
  $\mathrm{Fut}_\beta(M,\tau) = \mathrm{Fut}_D(M,\tau)$ for any
  $\mathbf{C}^*$ subgroup $\tau$ in $\mathbf{T}$. Let us consider the
  effect of varying the $\mathbf{C}^*$-action on the central fiber and
  the normalized twisted Futaki invariant. 

  As above, we can view the pair $(M,D)$ as a line spanned by a vector
  $v$ in a vector space $V$ with an action of $\mathbf{T}$. We 
  decompose $v = \sum v_{\alpha_i}$ into components on which the torus
  acts by weights $\alpha_i\in\mathfrak{t}^*$. Let us denote by
  $\mathcal{W}\subset\mathfrak{t}^*$ the weights that appear in this
  decomposition. For any
  $\mathbf{C}^*$-subgroup $\tau\subset \mathbf{T}$, we will also denote
  by $\tau\in\mathfrak{t}$ its generator. The central fiber
  $(M_0, D_0)$ under this $\mathbf{C}^*$ 
  is determined by the sum of those components 
  $v_{\alpha}$ for which $\langle\alpha, \tau\rangle$ is minimal,
  i.e. $\langle\alpha, \tau\rangle \leq \langle\beta,\tau\rangle$ for
  all $\beta\in\mathcal{W}$. Let us denote by
  $\mathcal{W}_\tau\subset\mathcal{W}$ the set of these minimal
  weights. The stabilizer of
  $(M_0,D_0)$ in $\mathbf{T}$ is then the subgroup with Lie algebra
  \[ \{ \eta\in \mathfrak{t}\,|\, \eta \text{ is constant on }
  \mathcal{W}_\tau\}, \]
  where we can view any $\eta\in\mathfrak{t}$ as a function on
  $\mathfrak{t}^*$. In particular the stabilizer of $(M_0,D_0)$ is
  $\tau$ precisely when $\mathcal{W}_\tau$ spans a codimension-one
  affine subspace in $\mathfrak{t}^*$. 

  Consider again our given special degeneration $\lambda$. If
  $\mathcal{W}_\lambda$ spans a codimension-one affine subspace, then
  we are already done. Otherwise, we can find another
  $\mathbf{C}^*$-action $\tau$ which is orthogonal to $\lambda$ in
  $\mathfrak{t}$ (here we use the inner product on $\mathfrak{t}$
  given by the $L^2$-product on $\mathbf{P}^N$ of the corresponding
  Hamiltonian functions), and is constant on
  $\mathcal{W}_\lambda$. For rational $t$ let us consider the
  $\mathbf{C}^*$-actions $\lambda + t\tau$. We can find an interval
  $(a_1,a_2)$ containing 0, such that if $t\in (a_1,a_2)$ then
  $\mathcal{W}_{\lambda + t\tau} = \mathcal{W}_\lambda$, 
  however for $i=1,2$ we have $\mathcal{W}_{\lambda + a_i\tau}
  \supsetneq \mathcal{W}_\lambda$. For $t\in (a_1,a_2)$ the central
  fibers $(M_0,D_0)$ of the degenerations given by $\lambda + t\tau$ will all be
  the same. As a result the twisted Futaki invariant varies
  linearly in $t$, while the norm is smallest when $t=0$. It follows
  that the normalized twisted Futaki invariant of $\lambda+t\tau$ will
  be strictly smaller for either $t=a_1$ or $t=a_2$ than for
  $t=0$. Moreover the original central fiber $(M_0,D_0)$ will be a
  specialization of the new $(M_0', D_0')$, and so $M_0'$ is also
  normal. The new central fiber has smaller stabilizer, and so after
  finitely many such steps the result follows. 
\end{proof}

\section{Proof of the main result}\label{sec:mainproof}
In this section we prove Theorem~\ref{thm:main}, along similar lines
to the argument in \cite{DSz}. Instead of the partial $C^0$-estimate
in \cite{Sz13_1}, we will use the main result in \cite{LiuSz}, which leads to
substantial simplifications, and allows us to work with non-negative
$\beta$ rather than just those that are strictly positive. We first
set up the relevant continuity 
method. 

\subsection{The continuity method}
Let $\alpha\in c_1(L)$ be a K\"ahler form, and consider the equations
\[ \label{eq:conteq}
  \mathrm{Ric}(\omega_t) = t\omega_t + (1-t)\alpha + \beta,
\]
for $\omega_t\in c_1(L)$.
For $t=0$ the equation can be solved using Yau's theorem~\cite{Yau78}, 
and the set of $t\in[0,1]$ for which the solution exists is open. Suppose that we
can solve the equation for $t\in [0,T)$. If $t > t_0 > 0$, then by
Myers' theorem we have a diameter bound, and since the volume is
fixed, the Bishop-Gromov theorem implies that the manifolds $(M,
\omega_t)$ are uniformly non-collapsed.  Along a sequence $t_k\to T$, we can
extract a Gromov-Hausdorff limit $Z$. Let us denote by $M_k$ the
metric spaces $(M,\omega_{t_k})$, so $M_k\to Z$ in the
Gromov-Hausdorff sense.  

Theorem 1.1 in \cite{LiuSz}
(which is based on ideas of Donaldson-Sun~\cite{DS1})
implies that for a sufficiently large $\ell > 0$, we have a sequence
of uniformly
Lipschitz holomorphic maps $F_k : M_k \to \mathbf{P}^N$, using sections
of $L^\ell$. These converge to a Lipschitz map $F_\infty : Z\to
\mathbf{P}^N$ that is a homeomorphism to its image. We will identify
$Z$ with its image $F_\infty(Z)$, which is a normal projective
variety. Up to choosing a
further subsequence we can assume that
\[ (F_k)_*[(1-t_k)\alpha + \beta] \to \gamma \]
weakly for a positive current $\gamma$ on $Z$. Note that
since the $F_k$ are all defined using sections of $L^\ell$, we have a
sequence $g_k\in PGL(N+1)$ such that $F_k = g_k\circ F_1$, so $Z$ is
in the closure of the $PGL(N+1)$-orbit of $F_1(M)$. 

We next show that $Z$ admits a twisted K\"ahler-Einstein
metric, which we can formally view as a solution of the equation
$\mathrm{Ric}(\omega_T) = T\omega_T + \gamma$. More precisely, let us
denote by $L$ the $\mathbf{Q}$-line bundle on $Z$ such that $L^l =
\mathcal{O}(1)$. We then have the following.
\begin{prop}\label{prop:limiteq}
  The $\mathbf{Q}$-line bundle $L$ over $Z$ admits a metric with
  locally bounded potentials with the following property.
  Locally on $Z_{reg}$, if the
  metric is given by $e^{-\phi_T}$, then its
  curvature form $\omega_{\phi_T}$ satisfies
  \[ \label{eq:te1} \omega_{\phi_T}^n = e^{-T\phi_T -\psi} \]
  in the sense of measures, where $\ddb\psi = \gamma$. Here $Z_{reg}$
  denotes the regular set of $Z$ in the complex analytic sense. 
\end{prop}
\begin{proof}
  The metric on (a power of) $L$ is obtained by the partial
  $C^0$-estimate, as a limit of metrics $h_k$ on $L\to M_k$ that have
  curvature $\omega_{t_k}$. More concretely, the partial
  $C^0$-estimate implies that under our embeddings $F_k : M_k\to
  \mathbf{P}^N$, the pullback of the Fubini-Study metric is uniformly
  equivalent to $h_k$. Using this we can extract a limit metric on
  $\mathcal{O}(1)|_Z$ which will also be uniformly equivalent to the
  restriction of the Fubini-Study metric.

  Let us now consider a point $p\in Z_{reg}$ and a sequence $p_k\in
  M_k$ such that $p_k\to p$ under the Gromov-Hausdorff convergence. 
  We have a holomorphic
  chart $z_i$ on a neighborhood of $p$, and using the maps $F_k$ this
  gives rise to charts $z_{ki}$ on neighborhoods of $p_k\in M_k$ for
  large $k$,  converging to $z_i$. Using these charts we can view the
  metrics $\omega_{t_k}$ as being defined on a fixed ball
  $B\subset\mathbf{C}^n$. By the gradient estimate for holomorphic
  functions, we have a uniform bound $\omega_{t_k} >
  C^{-1}\omega_{Euc}$. In addition, by \cite[Proposition 3.1]{LiuSz}
  we can assume (shrinking the charts if necessary) 
  that we have uniformly bounded K\"ahler potentials $\phi_{t_k}$ for the
  $\omega_{t_k}$. Let us denote by $\alpha_k, \beta_k$ the forms
  corresponding to $\alpha, \beta$ on $M$. Equation~\eqref{eq:conteq}
  implies that $\alpha_k, \beta_k$ have potentials $\psi_{\alpha_k},
  \psi_{\beta_k}$ satisfying the equation
    \[ \label{eq:okeq}  \omega_{t_k}^n = e^{-t_k\phi_{t_k} -
        (1-t_k)\psi_{\alpha_k} - \psi_{\beta_k}}, \]
  i.e.
  \[ \label{eq:Riceq1}
       \mathrm{Ric}(\omega_{t_k}) = t_k\omega_{t_k} + (1-t_k)\alpha_k
       + \beta_k. \]
  Our goal is to be able to pass this equation to the limit as
  $k\to\infty$, i.e. $t_k\to T$.

  Let us observe first that since $\alpha, \beta$ are fixed forms on
  $M$, using the lower bound $\omega_{t_k} > C^{-1}\omega_{Euc}$, 
  we have a uniform bound
  \[ \int_B \big[ (1-t_k)\alpha_k + \beta_k \big] \wedge
    \omega_{Euc}^{n-1} < C. \]
  It follows that we can take a weak limit
  \[ \gamma = \lim_{k\to\infty} (1-t_k)\alpha_k + \beta_k. \]
  From \eqref{eq:okeq}, and the lower bound for
  $\omega_{t_k}$ we have uniform upper bounds for
  $(1-t_k)\psi_{\alpha_k} + \psi_{\beta_k}$. These psh functions can
  also not converge to $-\infty$ everywhere as $k\to\infty$, since the
  volume of $B$ with respect to the metric $\omega_{t_k}$ is bounded
  above. It follows that up to choosing a subsequence we can extract a
  limit
  \[ (1-t_k)\psi_{\alpha_k} + \psi_{\beta_k} \to \psi, \text{ in }
    L^1_{loc}. \]
  We then necessarily have $\gamma = \ddb\psi$.

  Let $\kappa > 0$, and denote by $E_\kappa$ the set where the Lelong
  numbers of $\gamma$ are at least $\kappa$. By Siu's
  theorem~\cite{Siu} $E_\kappa$ is a subvariety in $B$.
  From \cite[Claim 4.3]{LiuSz}, and the
  subsequent argument, it follows that for any $q\not\in E_\kappa$, we
have $V_{2n} - \lim_{r\to 0} r^{-2n}\mathrm{vol}(B(q,r)) <
\Psi(\kappa)$, where the volume is measured using the limit metric on 
$Z$.   Here, and below, $\Psi(\kappa)$
  denotes a function converging to zero as $\kappa\to 0$, which may
  change from line to line.
In other words in the limit space $Z$ the complement of
$E_\kappa$  is contained in the $\epsilon$-regular set for $\epsilon =
\Psi(\kappa)$.
 
  Suppose now that $q\not\in E_\kappa$, and $\delta$ is sufficiently
  small so that $V_{2n} - \delta^{-2n}\mathrm{vol}(B(q,\delta)) <
  \epsilon$, where $V_{2n}$ is the volume of the Euclidean unit ball.
  Then we can apply Lemma~\ref{prop:biholder} below
  to see that on $B(q,\delta)$ the metrics $\omega_{t_k}$ are
  bi-H\"older equivalent to $\omega_{Euc}$. On these balls the
  K\"ahler potentials $\phi_{t_k}$ satisfy uniform gradient estimates
  with respect to $\omega_{t_k}$, since
  $\Delta_{\omega_{t_k}}\phi_{t_k} = n$, and so the $\phi_{t_k}$
  satisfy uniform H\"older bounds with respect to $\omega_{Euc}$. It
  follows from this that up to
  choosing a subsequence we can find a limit $\phi_{t_k} \to \phi_T$
  in $C^\alpha_{loc}(B \setminus E_\kappa)$, and $\phi_T$ is uniformly
  bounded on $B$. In particular for $\omega_T = \ddb\phi_T$,
  the measures $\omega_{t_k}^n$ converge weakly to $\omega_T^n$ on
  $B\setminus E_\kappa$. 

  To derive the required equation \eqref{eq:te1}, we note that on
  $B\setminus E_\kappa$ we have 
 \[e^{-(1-t_k)\psi_{\alpha_k} -
    \psi_{\beta_k}} \to e^{-\psi} \text{ in } L^1_{loc}.\]
 From the
  semicontinuity theorem of Demailly-Koll\'ar~\cite{DK} this follows
  if we bound the Lelong numbers of $\psi$, which will be the case if
  $\kappa$ is sufficiently small. It follows that on $B\setminus
  E_\kappa$ we have an equality of measures $\omega_T^n = e^{-T\phi_T
    -\psi}$, and since $E_\kappa$ has zero measure with respect to
  $\omega_T^n$, the equality holds on $B$ as well. 
\end{proof}

We used the following lemma in the argument.
\begin{lem}\label{prop:biholder}
  Suppose that $B(p,1)$ is a unit ball in a K\"ahler manifold with
  $\mathrm{Ric} \geq 0$, together with holomorphic coordinates $z_i$
  that give an $\epsilon$-Gromov-Hausdorff approximation of $B(p,1)$
  to the Euclidean unit ball $B(0,1)\subset \mathbf{C}^n$.
  There exists an $\alpha > 1 - \Psi(\epsilon)$ and $C > 0$ such that
  for $q, q'\in B(p,1/2)$ we have
  \[ d(q,q') \leq C|z(q) - z(q')|^\alpha. \]
  As above, $\Psi(\epsilon)$
  denotes a function converging to zero as $\epsilon\to 0$, which may  change from line to line.
\end{lem}
\begin{proof}
  We can assume that $z(p)=0$. It is enough to prove that for any
  $\delta > 0$, if $\epsilon$ is sufficiently small, then for all $k >
  0$ and $q\not\in B(p, 2^{-k})$, we have $|z(q)| >
  (2+\delta)^{-k}$. We prove this by induction.

  Suppose that we
  have shown that $|z| > (2+\delta)^{-k}$ outside of
  $B(p,2^{-k})$. Denote by $2^k B(p,2^{-k})$ the same ball scaled up
  to unit size. By Colding's volume convergence theorem~\cite{Col} and
  the Bishop-Gromov monotonicity, together with
  \cite[Theorem 2.1]{LiuSz}, we have holomorphic coordinates $w$ on
  this ball, giving a $\Psi(\epsilon)$-Gromov-Hausdorff approximation
  to the Euclidean unit ball. We can assume that $w(p)=0$. Let us also
  use the coordinates $z' = (2+ \delta)^kz$, which map our ball onto a
  region containing the Euclidean unit ball. Viewing $w$ as a function
  of $z$, the Schwarz lemma implies that $|w| \leq (1+\Psi(\epsilon))|z'|$
  on the unit $z'$-ball, and so in particular, using that $w$ is a
  Gromov-Hausdorff approximation, we have $|z'| \geq  (1
  -\Psi(\epsilon)) / 2$ outside of the ball $2^kB(p,
  2^{-k-1})$. Scaling back, this means that $|z| \geq
  (2+\Psi(\epsilon))^{-1} (2+\delta)^{-k}$ outside of $B(p,
  2^{-k-1})$.  We then just need to choose $\epsilon$ small enough to
  make $\Psi(\epsilon)<\delta$, and the inductive step follows. 
\end{proof}

\subsection{The Ding functional and the Futaki invariant}
We will next use the existence of a twisted K\"ahler-Einstein metric as in
Proposition~\ref{prop:limiteq} to deduce the vanishing of the twisted
Futaki invariant, and the reductivity of the automorphism group. 

Let $Z\subset\mathbf{P}^N$ be a normal variety, together with the following
additional data. We have a $\mathbf{Q}$-line bundle $L$ on $Z$ (a power of
which is just $\mathcal{O}(1)$), and a locally bounded metric
$e^{-\phi_0}$ on $L$. In addition we have a closed positive current $\gamma$
on $Z$. We say that these define a twisted K\"ahler-Einstein metric if
the conclusion of Proposition~\ref{prop:limiteq} holds, i.e. 
locally on $Z_{reg}$ we have the equation $\omega_{\phi_0}^n =
e^{-T\phi_0 - \psi}$, where $\ddb\psi = \gamma$. In terms of this we
can define the twisted Ding functional on the space of all metrics $e^{-\phi}$
with locally bounded potentials. Abusing notation slightly, we will
denote by $e^{-T\phi - \psi}$ the measure
\[ e^{-T\phi - \psi} = e^{-T(\phi-\phi_0)} \omega_{\phi_0}^n. \]
Note that while $\phi, \phi_0$ are only locally defined in terms of
trivializations of $L$,
$\phi-\phi_0$ is a globally defined bounded function on
$Z$. 

We have the Monge-Amp\`ere energy functional $E$, defined by its variation
\[ \delta E(\phi) = \frac{1}{V} \int_Z \delta\phi\, \omega_\phi^n, \]
where $V$ is the volume of $Z$ with respect to $\omega_\phi$, 
and we define the twisted Ding functional~\cite{Ding88} by
\[ \mathcal{D}(\phi) = -TE(\phi) - \log\left(\int_Z
  e^{-T\phi-\psi}\right). \]
The variation of $\mathcal{D}$ is
\[ \delta\mathcal{D}(\phi) = -TV^{-1}\int_Z \delta\phi\,\omega_\phi^n -
\frac{\int_Z -T(\delta\phi) e^{-T\phi-\psi}}{\int_Z
  e^{-T\phi-\psi}}, \]
and so the critical points satisfy
\[  \omega_\phi^n = C e^{-T\phi-\psi}. \]
Up to changing $\phi$ by addition of a constant, this is the twisted
KE equation as required. 

The convexity of the twisted Ding functional follows 
exactly Berndtsson's argument in \cite{Ber13} (see also \cite{DSz}), 
and so in particular if there is a critical
point, then $\mathcal{D}$ is bounded below. As in \cite{CDS,DSz}, the
key consequences of this convexity are the reductivity of the
automorphism group of $(Z,\gamma)$, and the vanishing of a twisted
Futaki invariant.

The reductivity of the automorphism group is a generalization of 
Matsushima's theorem for K\"ahler-Einstein metrics~\cite{Mat}
(see also \cite{Ber13, BBEGZ, BWN, CDS3,DervanSektnan}). 
Following \cite{DSz}, we define the Lie algebra stabilizer
of $(Z,\gamma)$, as a subalgebra of $\mathfrak{sl}(N+1,\mathbf{C})$ by
\[ \mathfrak{g}_{Z,\gamma} = \{ w\in H^0(TZ)\,:\, \iota_w
  \gamma=0\}. \]
We then have, following \cite{CDS3} (see also \cite[Proposition 7]{DSz})
\begin{prop}
  Suppose that $Z$ admits a twisted KE metric as above. Then
  $\mathfrak{g}_{Z,\gamma}$ is reductive. 
\end{prop}

Following
Chen-Donaldson-Sun~\cite{CDS} we also apply the convexity of the
twisted Ding functional to deduce the vanishing of 
a twisted Futaki invariant on $Z$. For this we consider the variation of
$\mathcal{D}$ along a 1-parameter group of automorphisms which fixes
the twisting current $\gamma$. If the automorphisms are generated by a
vector field $v$ with
Hamiltonian $\theta$, then the variation of $\phi$ is $\theta$, so we get
\[ \label{eq:1} \mathrm{Fut}_{T,\gamma}(Z,v) = -TV^{-1}\int_Z \theta
  \omega_\phi^n + T\frac{\int_Z \theta e^{-T\phi- \psi}}{\int_Z
    e^{-T\phi - \psi}}. \]
As a result we have the following. 
\begin{prop}\label{prop:Fut0}
Suppose that $Z$ admits a twisted KE metric as above, and let
$e^{-\phi}$ be a metric on $L$ with locally bounded
potentials. Suppose that $v$ is a holomorphic vector field on $Z$ with
a lift to $L$, such that the imaginary part of $v$ acts by isometries
on $L$, and so that $\iota_v\gamma = 0$. 
Let $\theta$ denote a Hamiltonian for $v$, i.e. $L_v\omega_\phi =
\ddb\theta$.  Then $\mathrm{Fut}_{T, \gamma}(Z,v)=0$, where
$\mathrm{Fut}_{T,\gamma}(Z,v)$ is defined as in \eqref{eq:1}. 
\end{prop}

As in \cite{DSz},  we need to relate this formula to the ``untwisted''
Donaldson-Futaki 
invariant. A new difficulty here is that the metric $\omega$ is not in
$c_1(Z)$, and so the Donaldon-Futaki invariant can not be expressed in terms of
the Ding functional. Instead we use the differential geometric formula
given in Proposition~\ref{prop:DFformula}.

Let $e^{-\phi}$ denote the restriction of the Fubini-Study metric to $L$
on $Z\subset\mathbf{P}^N$, and $\omega_\phi$ its curvature. We can use
a method similar to Ding-Tian~\cite{DT} to give a more differential
geometric formula for the twisted Futaki invariant. The vector field $v$ is
given by the restriction of a holomorphic vector field on
$\mathbf{P}^N$, and $\theta$ is the restriction to $Z$ of a smooth
function on $\mathbf{P}^N$. It follows that we have uniform bounds
$|\theta|, |\nabla\theta|, |\Delta\theta| < C$ on $Z_{reg}$, where we
are taking the gradient and Laplacian using the metric $\omega_\phi$
on $Z_{reg}$. In addition we have an upper bound
$\mathrm{Ric}(\omega_\phi) < C\omega_\phi$ on $Z_{reg}$, and so the
current $C\omega_\phi - [\mathrm{Ric}(\omega_\phi)-\gamma]$ is
positive for a sufficiently large constant $C$. 
\begin{prop}
We have the equality
\[ -TV^{-1}\int_Z \theta
  \omega_\phi^n + &T\frac{\int_Z \theta e^{-T\phi-\psi}}{\int_Z e^{-T\phi - \psi}} \\ &=
  -nV^{-1}\int_{Z} \theta(\mathrm{Ric}(\omega_\phi) - T\omega_\phi -
  \gamma)\wedge \omega_\phi^{n-1}. \]
\end{prop} 
\begin{proof}
 Let us define the (twisted) Ricci potential $u$ on $Z_{reg}$ by
\[ \label{eq:udef} e^{-T\phi-\psi -u} = \omega_\phi^n. \]
Interpreting this as an equality of metrics on $K^{-1}$ (on $Z_{reg}$) and taking
curvatures, we have
\[ \label{eq:ric2}
  T\omega_\phi + \gamma + \ddb u = \mathrm{Ric}(\omega_\phi). \]
Since the current $C\omega_\phi - [\mathrm{Ric}(\omega_\phi) -
\gamma]$ on $Z_{reg}$ is positive, we have $\ddb u \leq
C\omega_\phi$ on $Z_{reg}$. Since the singular set of $Z$ has
codimension at least 2, it follows from this that $u$ is bounded
below. Consider a resolution $\pi:\tilde{Z}\to Z$, and let $\eta$ be a
metric on $\tilde{Z}$. Let $\omega_\epsilon = \pi^*\omega_\phi +
\epsilon\eta$. Then $\omega_\epsilon$ gives a family of
 smooth metrics on $\tilde{Z}$ converging to $\pi^*\omega_\phi$ as
 $\epsilon \to 0$.  Let us denote the pullback of $u$ to $\tilde{Z}$
 by $u$ as well. We have $\ddb u \leq C\omega_\epsilon$
away from the exceptional set, and since $u$ is
bounded below, this inequality holds on all of $\tilde{Z}$. In
particular we have $\Delta_\epsilon u \leq Cn$.  Following
Ding-Tian~\cite{DT}, we integrate the inequality
\[ \int_{\tilde{Z}} \frac{\Delta_\epsilon u}{1 + (u - \inf u)}
\omega_\epsilon^n \leq C \]
by parts to obtain
\[ \int_{\tilde{Z}} \frac{ |\nabla u|_\epsilon^2}{(1 + (u-\inf u))^2}
\, \omega_\epsilon^n \leq C. \]
Letting $\epsilon\to 0$, we obtain the same estimate on $Z_{reg}$ with
the metric $\omega_\phi$. Just as in \cite{DT} we have that $u\in L^p$
for any $p$, and in turn this implies that we have a bound
\[ \int_{Z_{reg}} |\nabla u|^p \omega_\phi^n < C_p, \]
for any $p <2$. 

Differentiating the equation \eqref{eq:udef} along the vector field
$v$ we get that on $Z_{reg}$
\[ -T\theta - v(\psi) - v(u) = \Delta \theta. \]
Note that we can think of $v(\psi)$ as being defined by this equation
(since $\psi$ itself is only defined in local charts), since all other
terms are globally defined functions. In particular by the above
estimate for $u$ we have that $v(\psi)$ is in $L^p$ for $p < 2$. At
the same time, differentiating \eqref{eq:ric2}, and noting that
$L_v\gamma=0$, we get
\[ \ddb\big[ T\theta + v(u) + \Delta\theta\big] = 0, \]
and therefore we also have $\ddb v(\psi) = 0$. In particular
$\Lambda = v(\psi)$ is a constant on $Z$, and so
\[ \label{eq:udiff}
 - T\theta -\Lambda  = \nabla\theta\cdot \nabla u
  +\Delta\theta. \]

Since the integral 
\[ \int_{Z} e^{-T\phi - \psi} \]
is unchanged by flowing along the vector field $v$, we obtain
\[ \int_Z (-T\theta - \Lambda) e^{-T\phi-\psi} = 0. \]
Rearranging this, 
\[  \Lambda = -T\frac{ \int \theta e^{-T\phi- \psi}}{\int e^{-T\phi-\psi}}.
  \]
Using this formula in \eqref{eq:udiff}, and integrating, we get
\[ \label{eq:Fut2} -T\int \theta \omega_\phi^n + TV  \frac{ \int
  \theta e^{-T\phi-\psi}}{\int
  e^{-T\phi-\psi}} = \int (\nabla\theta\cdot
\nabla u + \Delta u) \omega_\phi^n, \]
where all integrals are on $Z_{reg}$. To integrate by parts, note that
since the singular set of $Z$ has real codimension at least 4, we can
find cutoff functions $\chi_\epsilon$ with compact support in
$Z_{reg}$ such that $\chi_\epsilon = 1$ outside the
$\epsilon$-neighborhood of $Z_{sing}$, and $\Vert \nabla\chi_\epsilon\Vert_{L^4} < C$. We
then have
\[ \int_{Z_{reg}} \nabla\theta\cdot \nabla u\, \omega_\phi^n &=
\lim_{\epsilon\to 0} \int \chi_\epsilon \nabla\theta\cdot\nabla
u\,\omega_\phi^n \\
&= \lim_{\epsilon\to 0} \left[ -\int \theta \nabla\chi_\epsilon\cdot
  \nabla u\,\omega_\phi^n - \int \chi_\epsilon \theta\Delta
  u\,\omega_\phi^n\right] \\
&= -\int \theta\Delta u\,\omega_\phi^n, \]
Here we used that $|\nabla u|\in L^{4/3}$, and so 
\[ \left| \int \theta\nabla\chi_\epsilon\cdot\nabla u\,\omega_\phi^n
\right| \leq C \Vert\nabla\chi_\epsilon\Vert_{L^4}
\left(\int_{\mathrm{supp}(\nabla\chi_\epsilon)} |\nabla
  u|^{4/3}\,\omega_\phi^n\right)^{3/4} \to 0\,\text{ as }\epsilon\to
0. \]
Similarly we can check that $\int \Delta u\,\omega_\phi^n = 0$. 
In conclusion, from \eqref{eq:Fut2} we find that
\[ -TV^{-1}\int \theta \omega_\phi^n + T \frac{ \int
  \theta e^{-T\phi -    \psi}}{\int
  e^{-T\phi - \psi}} = -nV^{-1}\int_{Z_reg}
\theta(\mathrm{Ric}(\omega_\phi) - T\omega_\phi - 
  \gamma )\wedge \omega_\phi^{n-1},\]
as required. 
\end{proof}

Suppose now that $Z$ is the central fiber of a special degeneration
for $M$ induced by the one-parameter group $\lambda(t)$. Then
using Proposition~\ref{prop:DFformula}, we can relate the twisted
Futaki invariant to the Donaldson-Futaki invariant as follows. 
\begin{cor}\label{cor:Fut2}
  The twisted Futaki invariant above is given by
  \[ \mathrm{Fut}_{T,\gamma}(Z, v) = DF(M,\lambda) + nV^{-1}\int_Z
    \theta(\gamma - c\omega_\phi)\wedge \omega_\phi^{n-1}, \]
  where $\lambda$ is a $\mathbf{C}^*$-action generated by the vector
  field $v$, and $c$ is a constant so that the right hand side is
  unchanged when we add a constant to the Hamiltonian $\theta$. 
\end{cor}

\subsection{Completion of the proof of Theorem~\ref{thm:main}}
We can now complete the proof of the main result. According to
Corollary~\ref{cor:notKstable} it is enough to show that either we can
find special degenerations for $M$ with arbitrarily small twisted
Futaki invariant, thereby contradicting the K-stability of
$(M,\beta)$, or $T=1$ and the twisted KE metric that we obtained on
$Z$ is actually the twisted KE metric on $M$ that we set out to find.

Let us denote by $Z\subset\mathbf{P}^N$
the Gromov-Hausdorff
limit of $(M,\omega_{t_k})$ along the continuity path
\eqref{eq:conteq}. Using
Proposition~\ref{prop:limiteq} we know that $Z$ admits a twisted KE
metric. In particular the pair $(Z,\gamma)$ is
in the closure of the $PGL(N+1)$-orbit of $(M, (1-T)\alpha + \beta)$,
where $T=\lim t_k$, and we are identifying $M$ with its image $F_1(M)$. We can now
closely follow the method in \cite{DSz} of approximating the forms $\alpha, \beta$
by currents of integration along divisors in $M$. Just like in
\cite{DSz}, the twisted Futaki invariants become smaller as $T$
increases (see \cite[Equation (23)]{DSz}). Because of this, and to
simplify the discussion below, we will assume that $T=1$. Note that unlike
the setting in \cite{DSz}, here we still have a twisting term when
$T=1$,  and so this case is not any easier than the case $T<1$. 

By assumption, the form $\beta$ on $M$ can be written as an integral of
currents of integration, as in Equation~\eqref{eq:bdecomp}. 
Recall also that we have the sequence $g_k\in PGL(N+1)$ such that $F_k
= g_k\circ F_1$, and so $g_k(M)\to Z$. As in \cite[Lemma 14]{DSz}, by
choosing a subsequence we can 
ensure that each sequence $g_k(D)$ for $D\in |T|$ converges to a subvariety of
$\mathbf{P}^N$ which we denote by $g_\infty(D)$. It follows that we
have
\[ (g_k)_*\beta \to \int_{|T|} [g_\infty(D)]\,
  d\mu(D), \]
in the weak topology. The twisting current
$\gamma$ on $Z$ is obtained as the limit of $(g_k)_*\beta$ as
$k\to\infty$, and so we have
\[ \gamma = \int_{|T|} [g_\infty(D)] \,\mu(D). \]

Arguing as in \cite[Lemma 15]{DSz}, we can find a finite set
$D'_1,\ldots, D'_r\in |T|$ such that the Lie algebra of the stabilizer of
the tuple $(Z, g_\infty(D'_1),\ldots, g_\infty(D'_r))$ in $PGL(N+1)$ is
$\mathfrak{g}_{Z, \gamma}$, and in particular it is reductive. In
addition there is a subset $E\subset |T|$ of measure zero such that if
$D_1,\ldots, D_K\not\in E$, then  the
stabilizer of the extended tuple $(Z, g_\infty(D'_1),\ldots,
g_\infty(D'_r), g_\infty(D_1),\ldots,g_\infty(D_K))$ is still
reductive. Suppose that this tuple is not in the $PGL(N+1)$-orbit of
$(M,D_1',\ldots,D_r', D_1,\ldots, D_K)$. Then we can find a $\mathbf{C}^*$-subgroup
$\lambda_K\subset PGL(N+1)$ and an element $g_K\in PGL(N+1)$ such that
\[ Z&= \lim_{t\to 0} \lambda_K(t) g_K\cdot M,\\
  g_\infty(D'_i) &= \lim_{t\to 0} \lambda_K(t) g_K\cdot D'_i, \text{ for }
  i=1,\ldots, r,\\
  g_\infty(D_j) &= \lim_{t\to 0} \lambda_K(t) g_K\cdot D_j, \text{ for }
  j=1,\ldots, K. \]
Suppose that $\lambda_K$ is generated by a vector field $w_K$, with
Hamiltonian $\theta_K$, and we normalize $\theta_K$ so that it has
zero average on $\mathbf{P}^N$. 
In addition we can scale $w_K$ so that $\Vert \theta_K\Vert_{L^2} = 1$. Note
that since $Z$ is not contained in a hyperplane, the Hamiltonian
$\theta_K$ cannot be constant on $Z$, unless $\lambda_K$ is
trivial. 

We can choose $D_1,\ldots,D_K\in |T|\setminus E$ so that no $d+1$ lie on
a hyperplane in $|T|$. Here $d$ is the dimension
of the projective space $|T|$. From
Proposition~\ref{prop:gw} we have
\[ \lim_{t\to 0} \int_{\lambda_K(t)g_K\cdot M} \theta_K\,
  (\lambda_K(t)g_K)_*\beta\wedge 
  \omega_{FS}^{n-1} &= \frac{1}{K}\sum_{i=1}^K \lim_{t\to 0}
  \int_{\lambda_K(t)g_K\cdot D_i} \theta_K\,\omega_{FS}^{n-1} +
  O(1/K) \\
  &= \frac{1}{K} \sum_{i=1}^K \int_{g_\infty(D_i)}
  \theta_K\,\omega_{FS}^{n-1} + O(1/K), \]
since $d$ is independent of $K$.

At the same time given any $\epsilon > 0$ we can choose $K$ large and
the $D_i$ so that
\[ \frac{1}{K} \sum_{i=1}^K \int_{g_\infty(D_i)}
  \theta_K\,\omega_{FS}^{n-1} \leq \int_Z
  \theta_K\,\gamma\wedge\omega_{FS}^{n-1} + \epsilon. \]
Let us denote by $\gamma_K = \lim_{t\to 0} (\lambda_K(t)g_K)_*\beta$
the limit current on $Z$. 
Combining our inequalities, and the assumption of twisted K-stability, we have
\[ 0 &\leq \mathrm{Fut}_\beta( g_K\cdot M, \lambda_K) = DF(Z,
  \lambda_K) + nV^{-1}\int_Z \theta_K\,(\gamma_K-c\omega_{FS})\wedge
  \omega_{FS}^{n-1} \\
  &= DF(Z, \lambda_K) + nV^{-1}\frac{1}{K} \sum_{i=1}^K \int_{g_\infty(D_i)}
  \theta_K\,\omega_{FS}^{n-1} - cnV^{-1}\int_Z \theta_K\,
  \omega_{FS}^n + O(1/K) \\
  &\leq DF(Z, \lambda_K) + nV^{-1}\int_Z
  \theta_K\,(\gamma-c\omega_{FS})\wedge\omega_{FS}^{n-1} + \epsilon + O(1/K)\\
  &= \epsilon + O(1/K). \]
Note that in the last line we used Proposition~\ref{prop:Fut0} and
Corollary~\ref{cor:Fut2}. Choosing $\epsilon$ small and $K$
sufficiently large, it follows that if the tuples $(Z, g_\infty(D_i'),
g_\infty(D_j))_{i=1,\ldots,r, j=1,\ldots, K}$ are not in the
$PGL(N+1)$-orbit of $(M, D_i', D_j)_{i=1,\ldots,r,j=1,\ldots,K}$ for
infinitely many $K$, then we have special degenerations for
$(M, \beta)$ with arbitrarily small twisted Futaki
invariant. Corollary~\ref{cor:notKstable} then implies that
$(M,\beta)$ is not K-stable. 

Otherwise, $Z$ is in the $PGL(N+1)$-orbit of $M$, and since under our
assumptions $M$ has discrete stabilizer group, it follows that the
group elements $g_k$ are uniformly bounded. As in \cite{DSz}, this
implies that the solutions $\omega_{t_k}$ along the continuity method
satisfy uniform estimates, and so we obtain a solution for $t=T$ as
well, as required.

\bibliography{twistedKE}
\bibliographystyle{amsplain}

\end{document}